\documentclass[12pt,reqno]{article}

\usepackage[usenames]{color}
\usepackage{amssymb}
\usepackage{graphicx}
\usepackage{amscd}

\usepackage[colorlinks=true,
linkcolor=webgreen,
filecolor=webbrown,
citecolor=webgreen]{hyperref}

\definecolor{webgreen}{rgb}{0,.5,0}
\definecolor{webbrown}{rgb}{.6,0,0}

\usepackage{color}
\usepackage{fullpage}
\usepackage{float}

\usepackage{graphics,amsmath,amssymb}
\usepackage{amsthm}
\usepackage{amsfonts}
\usepackage{latexsym}

\begin{document}

\vspace*{2.1cm}

\theoremstyle{plain}
\newtheorem{theorem}{Theorem}
\newtheorem{corollary}[theorem]{Corollary}
\newtheorem{lemma}[theorem]{Lemma}
\newtheorem{proposition}[theorem]{Proposition}
\newtheorem{obs}[theorem]{Observation}
\newtheorem{claim}[theorem]{Claim}

\theoremstyle{definition}
\newtheorem{definition}[theorem]{Definition}
\newtheorem{example}[theorem]{Example}
\newtheorem{remark}[theorem]{Remark}
\newtheorem{conjecture}[theorem]{Conjecture}

\begin{center}

{\Large\bf A generalization of perfect codes in the presence of star multiset transpositions} 

\vskip 1cm
\large
Italo J. Dejter

University of Puerto Rico

Rio Piedras, PR 00936-8377

\href{mailto:italo.dejter@gmail.com}{\tt italo.dejter@gmail.com}
\end{center}

\begin{abstract}
Let $0<\ell\in\mathbb{Z}$. The notion of an efficient dominating set or perfect code $S$ of a graph $G$ is generalized to that of an efficient dominating$\,^\ell$-set or perfect$^\ell$code, of the graph $G$, meaning that each vertex $v$ of $V(G)\setminus S$ has exactly $\ell$ neighbors in $S$, instead of just one neighbor. Such generalization is applied to star $j$-set transposition graphs based on permutations of multisets with each symbol repeated $j$ times, ($j\in\{\ell,\ell-1\}$). In such vertex-transitive graphs this approach produces total colorings, efficient dominating sets, also called perfect codes, etc. 
\end{abstract}

\section{Introduction}\label{s1}

Let $0<\ell\in\mathbb{Z}$, let $G=(V(G),E(G))$ be a finite graph of girth larger than 3 and let $S\subseteq V(G)$. We say that $S$ is an {\it efficient dominating$\,^\ell\!$-set} (E$\,^\ell$-set) or a {\it perfect$\,^\ell\!$code} if for each $v\in V(G)\setminus S$ there are exactly $\ell$ vertices $v^0,v^1,\ldots,v^{\ell-1}$ in $S$ such that $v$ is adjacent to  $v^i$, for each $i\in[\ell]=\{0,\ldots,\ell-1\}$. Since we assume that the girth of $G$ is larger than 3, then the subset $S(v)=\{v^0,v^1,\ldots,v^{\ell-1}\}$ of $V(G)$ is an independent set of $G$ that we call a  {\it dominating $\ell$-set} (D$\ell$S) of $v$ {\it with respect to} $S$ ({\it wrt} $S$). In particular: 

\begin{enumerate}\item[\bf(a)] if $\ell=1$, then $S$ is an {\it efficient dominating set} 
(E-set) \cite{AK,BBS,gen,D73,D76,edig}, or {\it 1-perfect code} \cite{Borges,D80}; in this case, $S$ provides a perfect packing of $G$ by balls of radius 1, or {\it 1-spheres}; if $G$ is $r$-regular,  the sphere packing condition  $|V(G)|=(r+1)|S|$ is a necessary condition for $S$ to be an E-set of $G$ \cite{D73};
\item[\bf(b)] if $\ell>1$, then $v$ is the only vertex of $G$ in the intersection of the 1-spheres of  $v^0,v^1,\ldots,v^{\ell-1}$; we deal with this extension case of an E-set from Section~\ref{Galphak} on.\end{enumerate}  

In Section~\ref{s2}, we recall the definition of the vertex-transitive star $\ell$-set transposition graphs $G =ST^\ell_k$, ($1<k\in\mathbb{Z}$), 
a particular case of star multiset transposition graphs treated in \cite{faltaba} in a context of determining their Gray codes and Hamilton paths and cycles.

In Sections~\ref{Galphak} and \ref{s4}, we find E$\,^\ell$-sets and E$\,^{\ell-1}$-sets of the graphs $ST^\ell_k$, respectively, by adapting the setting of E-chains in \cite{D73} with applications to total colorings of graphs, error-correcting codes and networks; (see also Remark~\ref{appl}).

\section{Star $\ell$-set transposition graphs}\label{s2}

Let $0<\ell\in\mathbb{Z}$ and $2\le k\in\mathbb{Z}$. We say that a string over the alphabet $[k]=\{0,\ldots,k-1\}$ that contains exactly $\ell$ occurrences of $i$, for each $i\in[k]$, is an $\ell$-{\it set permutation}.
In denoting specific $\ell$-set permutations, commas and brackets will be usually omitted. 
Let $V^\ell_k$ be the set of all $\ell$-set permutations of length $k\ell$. 

Let $ST^\ell_k$ be the graph on vertex set $V^\ell_k$ with an edge between any two vertices $v=v_0v_1\cdots v_{k\ell-1}$ and $w=w_0w_1\cdots w_{k\ell-1}$ that differ in a {\it star $\ell$-set transposition}, meaning that it is obtained by swapping the first entry $v_0$ of $v=v_0v_1\cdots v_{k\ell-1}\in V^\ell_k$ with some entry $v_j$, ($j\in[k\ell]\setminus\{0\}$), whose value differs from that of  $v_0$, so $v_j\ne v_0$, thus obtaining either $w=w_0\cdots w_j\cdots w_{k\ell-1}=v_j\cdots v_0\cdots w_{k\ell-1}$ or $w=w_0\cdots w_{k\ell-1}=v_{k\ell-1}\cdots v_0$. 

Note that $ST^\ell_k$ has $\frac{(k\ell)!}{(\ell!)^k}$ vertices and regular degree $(k-1)\ell$.

It is known that all 1-set permutations of length $k$ (usually known as  $k$-permutations) form the {\it symmetric group} denoted $Sym_k$ under composition of
$k$-permutations, each $k$-permutation $v=v_0v_1\cdots v_{k-1}$ taken as a bijection from the {\it identity} $k$-permutation $01\cdots(k-1)$ onto $v$ itself. 
A graph $ST^1_k$ with $k>1$ (which excludes $ST^1_1$) is the Cayley graph of $Sym_k$ with respect to the set of transpositions $\{(0\;i); i\in[k]\setminus\{0\}\}$.
Such $ST_k^1$ is denoted $ST_k$ in \cite{AK, D73}, where is proven that its vertex set admits a partition into $k$ E-sets, exemplified on the upper left of Figure~\ref{fig1}, for $ST^1_3=ST_3$, with the vertex parts of the partition differentially colored in black, red and green, for first entries 0, 1 and 2, respectively. Figure 1 of \cite{D73} shows a similar example for $ST^1_4=ST_4$. Note that the graphs $ST^\ell_k$ are vertex-transitive, but they are neither Cayley or Shreier graphs, for $\ell>1$.

\section{E$\,^\ell$-sets of star $\ell$-set transposition graphs}\label{Galphak}

The vertices of $ST^\ell_k$ can be seen as the $\ell$-set permutations $v_0\dots v_{k\ell-1}$ of the string $$\overbrace{0\cdots 0}\overbrace{1\cdots 1}\overbrace{2\cdots 2}\cdots\overbrace{(k-1)\cdots(k-1)}=0^\ell 1^\ell2^\ell\cdots(k-1)^\ell.$$         
Let us see, for each $i\in[k]$ with $ST^\ell_k\ne ST^1_2$, that the vertices $v=v_0\dots v_{k\ell-1}$ of $ST^\ell_k$ with first entry $v_0$ equal to $i$ form an E$\,^\ell$-set $S_i^k$ of $ST^\ell_k$. Since $ST^\ell_k\ne ST^1_2$ has girth larger than 3, then each D$\ell$S wrt $S_i^k$ of $ST^\ell_k$ is an independent set of $ST^\ell_k$.    
Extending the notion of an E-chain in \cite{D73}, we say that an E$\,^\ell$-chain is a countable family of nested graphs, each of which has an E$\,^\ell$-set.

\begin{figure}[htp]
\includegraphics[scale=0.88]{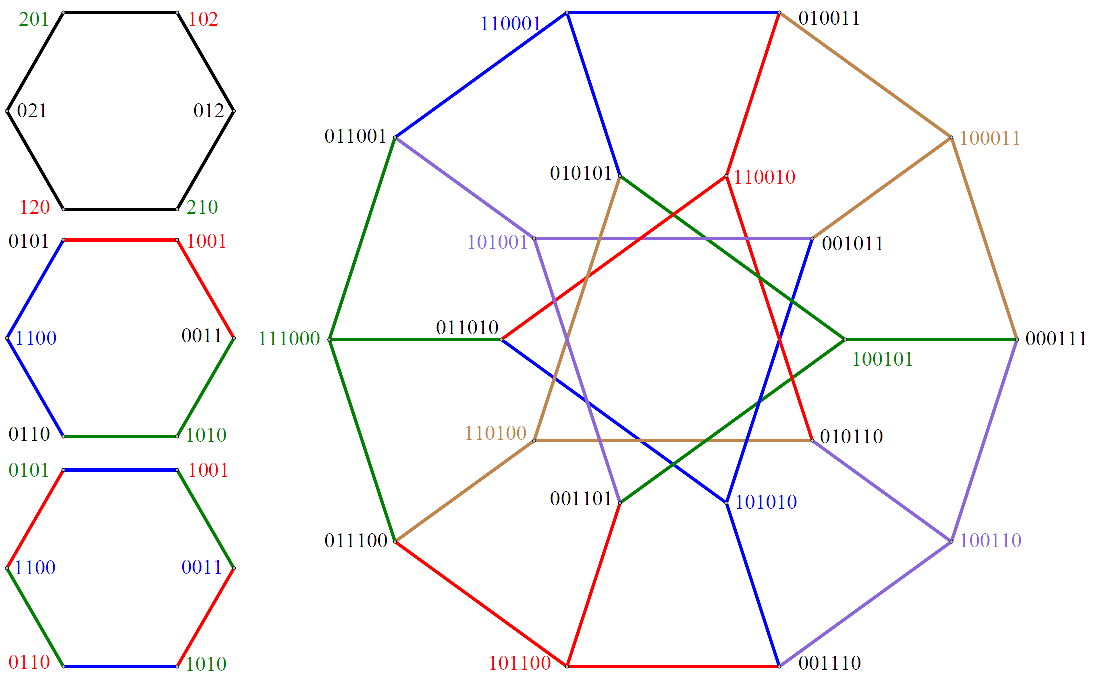}
\caption{The 6-cycles $ST^1_3=ST_3$ and $ST^2_2$, and the Desargues graph $ST^3_2$.}
\label{fig1}
\end{figure}

\begin{theorem}\label{teo1}\label{t} Let $0<\ell,k\in\mathbb{Z}$ and let $ST^\ell_k\ne ST^1_2$.
For each $i\in[k]$, the $\ell$-set permutations $v_0\dots v_{k\ell-1}$ of  
$0^\ell 1^\ell2^\ell\cdots(k-1)^\ell$ with first entry $v_0$ equal to $i$ form
an E$\,^\ell$-set $S_i^k$ of $ST^\ell_k$. In particular, the graphs $ST^\ell_k$ form an E$\,^\ell$-chain.
\end{theorem}

\begin{proof}
For fixed $i\in[k]$, each vertex $v=v_0v_1\cdots,v_{k\ell-1}$ of $E(ST^\ell_k)\setminus S_i^k$ has initial entry $v_0=j$, for some $j\in[k]$ such that $j\ne i$. Then, $v$ is adjacent to $\ell$ vertices of $S_i^k$ obtained by transposing the position of each of the $\ell$ entries $v_h=i$, ($h\in[k\ell]\setminus\{0\}$), with the position of that initial entry $v_0=j$. The graph induced by the resulting edges form: {\bf(a)} a copy $H$ of the complete bipartite graph $K_{1,\ell}$ (a 1-sphere) with $v$ as its sole degree-$\ell$ vertex (center of such 1-sphere); and {\bf(b)} the leaves of $H$ (this taken as a $v$-rooted tree) as the vertices $v^j\in S_i^k$
obtained from $v$ by transposing $v_0=i$ with those $v_h=j$. 
\end{proof}

\begin{corollary}
 The vertex set $V(ST^\ell_k)$ admits a partition into $k$ E$\,^\ell$-sets $S_i^k$, ($i\in[k]$). 
\end{corollary}

\begin{proof} The E$\,^\ell$-sets $S_i^k$ form a partition of $V(ST^\ell_k)$, since each such $S_i^k$ is composed precisely by the vertices $v=v_0v_1\cdots v_{k\ell-1}$ of $ST^\ell_k$ having initial entry $v_0=i$, which form one of the $k$ parts of the partition.
\end{proof}

Let $2ST^\ell_k$ be the multigraph obtained from $ST^\ell_k$ by replacing each edge $e$ of $ST^\ell_k$ by two parallel edges with the same endvertices of $e$. 

\begin{corollary}
Let $i\in[k]$. Each vertex of $S_i^k$ in $ST^\ell_k$ belongs to $(k-1)\ell$ D$\,\ell$Ss wrt $S_i^k$, where the induced graph of each such D$\,\ell$S is isomorphic to the complete bipartite graph $K_{1,\ell}$. 
The set of all such D$\ell$Ss, $\forall i\in[k]$, forms a partition of the edge set of $2ST^\ell_k$.
\end{corollary}

\begin{proof}
Since there are $k-1$ values $j\ne i$ in $[k]$, each vertex $v\in S_i^k$ belongs to $(k-1)\ell$ D$\ell$Ss wrt $S_i^k$. Since each such $v$ is the neighbor of $\ell$ vertices with first entry $v_0=j$, for each $j\in[k]$ with $j\ne i$, then the induced graph of each such D$\ell$S is isomorphic to $K_{1,\ell}$\,, as said in the proof of Theorem~\ref{t}. Also, each edge $e$ of $ST^\ell_k$ with endvertices $v$ and $w$ having first entries $i$ and $j$, respectively, is both a member of $S_i^k$ and of $S_j^k$. Thus, $e$ appears in $2ST^\ell_k$ as two parallel edges, edge with one endvertex in $S_i^k$ and another endvertex in $S_j^k$.
\end{proof}

\begin{example}\label{ex1} The graph $ST^2_2$ consists of the 6-cycle $(0011,1001,0101,1100,0110,1010)$ represented in the middle left of Figure~\ref{fig1} and showing in distinct edge shades the induced subgraphs of the composing D$2$Ss of the E$^2$-set $S_0^2=\{0011,0101,0110\}$, namely the D$2$S $\{0011,0101\}$ of 1001 wrt $S_0^2$, the D$2S \,$\{0011,0110\} of 1010 wrt $S_0^2$, and the D$2$S $\{0110,0101\}$ of 1100 wrt $S_0^2$.\end{example}

\begin{example}\label{ex2} The graph $ST^3_2$ is the Desargues graph drawn on the right of Figure~\ref{fig1}, where each subgraph $K_{1,3}$ induced by a D$3$Ss $S_0^2(v)$ of a vertex $v=0v_1\cdots v_5$ of $ST^3_2$ wrt $S_0^2$ has its edges in three different colors. The edge colors in this representation of $ST^3_2$ are seen to define ten monochromatic copies of $K_{1,3}$ centered at the vertices of the form $w=1w_1\cdots w_5$ in colors red, blue, green, hazel and violet (two vertex-disjoint monochromatic copies of $K_{1,3}$ per color), illustrating that each $v\in S_0^2$ is the intersection of $\ell=$ three 1-spheres, each 1-sphere contributing just one edge incident to $v$, where the three resulting edge colors are pairwise different.  
Those monochromatic subgraphs $K_{1,3}$ appear in ``opposing'' pairs, allowing to raise the question of how many colors are necessary to color such edge partitions. 
\end{example}

\begin{example}\label{ex3}
The vertices of $ST^2_3$ are the $\ell$-set permutations, ($\ell=2$), of $v_0^0=001122$, yielding a total of $\frac{6!}{2!2!2!}=\frac{720}{8}=90$ vertices. The regular degree of $ST^2_3$ is 4.
The graph $ST^2_3$ has the E$^3$-set $S_0^2=\{v_0\cdots v_5\in V(ST^2_3);v_0=0\}$. 
For example, $v=100122$ has $S^k_0(v)=S^2_0(v)=\{010122,001122\}$ as its D$\ell$S wrt $S$, 
and $v'=120120$ has $S^k_0(v')=S^2_0(v')=\{021120,020121\}$ as D$\ell$S wrt $S_0^2$. Each vertex $v$ in $S_0^2$ belongs to four D$2$Ss wrt $S_0^2$. 
While $ST^2_3$ has 90 vertices, $S_0^2$ has $\frac{90}{k}=\frac{90}{3}=30$ vertices.
For example, 010122 belongs to $S^2_0(100122)$, $S^2_0(110022)$, $S^2_0(210102)$ and $S^3_0(210120)$. Specifically, as in display (\ref{ccc}).
\begin{eqnarray}\label{ccc}\begin{array}{ccc}
S_0^3(100122)=\{010122,001122\},\\
S_0^3(110022)=\{010122,011022\},\\
S_0^3(210102)=\{010122,012102\},\\
S_0^3(210120)=\{010122,012120\}.\end{array}\end{eqnarray}\end{example}

\begin{example}\label{ex4}
The vertices of $ST^3_3$ are the $3$-set permutations of $v_0^0=000111222$, yielding a total of $|V(ST^3_3)|=\frac{(k\ell)!}{\ell!^k}=\frac{9!}{3!^3}=1680$ vertices. The regular degree of $ST^3_3$ is six.
The graph $ST^3_3$ has the set of 9-tuples $S_0^3=\{v_0\cdots v_8\in V(ST^3_3);a_0=0\}$ as an E$^3$-set. For example, $100011222$ has $\{000111222,001011222,010011222\}$ as its D$3$S wrt $S_0^3$, and $120120120$ has $\{021120120,020121120,020120121\}$ as its D$3$S wrt $S_0^3$.
While $ST^3_3$ has $1680$ vertices, $S_0^3$ has $\frac{1680}{k}=\frac{1680}{3}=560$ vertices. Each vertex of $S_0^3$ belongs to six D$3$Ss wrt $S_0^3$. For example, 010011222 belongs to the sets in display (\ref{c}).
\begin{eqnarray}\label{c}\begin{array}{c}S_0^3(100011222)=\{010011222,001011222,000111222\},\\
S_0^3(110001222)=\{010011222,010101222,011001222\},\\
S_0^3(110010222)=\{010011222,010011222,010001222\},\\
S_0^3(210011022)=\{010011222,010211022,012001022\},\\
S_0^3(210011202)=\{010011222,010211202,012001202\},\\
S_0^3(210011220)=\{010011222,010211220,012001200\}.\end{array}\end{eqnarray}\end{example}

\begin{corollary}
Let $i\in[k]$ and let $S_i^k$ be an E$\,^\ell$-set of $ST^\ell_k$. For each fixed vertex $v=v_0v_1\cdots v_{k\ell-1}\in V(ST^\ell_k)\setminus S_i^k$, the $\ell$ vertices of the $D\ell$S $S_i^k(v)$ (wrt $S_i^k$) bear bijectively the $\ell$ occurrences of $i$, each such occurrence as a value of a corresponding non-initial entry $v_h$ of $v$, ($h\in\{1,\ldots,k\ell-1\}$), transposed with the value $j$ at its initial entry.  
\end{corollary}

\begin{proof}
The behavior in the statement is exemplified in Examples~\ref{ex3}-\ref{ex4} (displays (\ref{ccc})-(\ref{c}), respectively). Such behavior generalizes similarly for larger values of $\ell$ and/or $k$.
\end{proof}

\begin{remark}\label{appl} If $ST^\ell_k$ is taken as the plan map of a city with streets represented by edges and corners represented by vertices, then an E$\,^\ell$-set $S_i^k$ may be planned to hold cops stationed at its vertices. In the case of an  event at a vertex $v$ of $ST^\ell_k$, if the vertex is in $S_i^k$, then a corresponding cop is at hand. Otherwise, there are $\ell$ cops at the vertices in $S_i^k(v)$, any of which can be present by moving along one sole edge. As another application, an error-correcting model of $ST^\ell_k$ will give for each received message a total of 1 (in $S_i^k$) or $\ell$ (in $ST^\ell_k\setminus S_i^k$) corrected messages.\end{remark}

\section{E$\,^{\ell-1}$-sets of star $\ell$-set transposition graphs}\label{s4}

Let $1<\ell\in\mathbb{Z}$, let $i\in[k\ell]\setminus\{0\}$, let $\Sigma_i^k$ be the set of vertices $v_0v_1\cdots v_{k\ell-1}$ of $ST^\ell_k$ such that $v_0=v_i$, ($i\in[k\ell]\setminus\{0\}$), and 
let $E_i^k$ be the set of edges having color $i$ in $G\setminus\Sigma_i^k$.
We will show that $\Sigma_i^k$ is an E$\,^{\ell-1}$-set of $ST^\ell_k$. Clearly, no edge of $E_i^k$ is incident to the vertices of $\Sigma_i^k$.

In order to state Theorem~\ref{t2}, let us recall that a {\it total coloring} of a graph $G$ is an assignment of colors to the vertices and edges of $G$ such that no two incident or adjacent elements (vertices and/or edges)
are assigned the same color \cite{tc-as}. 
A total coloring of $G$ such that the vertices adjacent to each $v\in V(G)$ together with $v$ itself are assigned pairwise different colors will be said to be an {\it efficient coloring}. This will be said to be {\it totally efficient} if: {\bf(a)} $G$ is $r$-regular; {\bf(b)} the color set is $[r]=\{0,1,\ldots,r-1\}$; and {\bf(c)}
each $v\in V(G)$ together with its neighbors are assigned all the colors in $[r]$. Now, let: {\bf(c')} each $v\in V(G)$ together with its neighbors are assigned all but one colors in $[r]$. 
If items (a) and (b) hold as above but (c) is replaced by (c'), then we say that the coloring is {\it almost totally efficient}.

\begin{theorem}\label{t2}  {\bf(I)} Let $k>1$, let $i\in[k\ell]\setminus\{0\}$ and
let $\Sigma_i^k$ be the set of vertices $v_0v_1\dots v_{k\ell-1}$ of $ST^\ell_k$ such that $v_0=v_i$. Then, $V^\ell_k$ admits a vertex partition into $k\ell-1$ E$\,^{\ell-1}$-sets $\Sigma_i^k$, ($i\in[k\ell]\setminus\{0\}$). In particular, the graphs $ST^\ell_k$ form an E$\,^{\ell-1}$-chain. {\bf(II)} 
Let $k>2$, let $j\in[k\ell]\setminus\{0\}$ and let
$E_j^k$ be the set of all edges of color $j$. Then, $ST^\ell_k\setminus\Sigma_i^k\setminus E_i^k$ is the disjoint union of 
$k\ell^{k-1}$ copies of $ST^\ell_{k-1}$. {\bf(III)} If $\ell=2$, then the objects presented in items {\rm(I)-(II)} yield a totally efficient coloring of $ST^\ell_k$, with each $\Sigma_i^k$ being an E-set, that is a perfect code of $ST^\ell_k$. In this case, $ST^\ell_k$ admits a vertex partition into E-sets. {\bf(IV)} If $\ell=3$, then the objects presented in items (I)-(II) yield an almost totally efficient coloring of $ST^\ell_k$.
\end{theorem}

\begin{proof}
The graph $ST^\ell_k$ has $\frac{(k\ell)!}{(\ell!)^k}$ vertices and regular degree $(k-1)\ell$.
Let $i=k\ell-1$ and let $i\in[k\ell]$. Assume without loss of generality that $i=k\ell-1$, corresponding to the last entry of the vertices of $ST^\ell_k$. Then, 
each vertex $v$ of $ST^\ell_k$ with differing first and last entry values is adjacent in $\ell-1$ different ways to a corresponding vertex of $\Sigma_i^k=\Sigma_{k\ell-1}^k$, yielding a total of $\ell-1$ such vertices of $\Sigma_i^k$. Thus, $\Sigma_i^k$ is an E$\,^{\ell-1}$-set. 
Note that the edge obtained by transposing the values of the first and last entries of $v$ is an edge of $E_i^k=E_{k\ell-1}^k$, and that all edges of such $E_i^k$ are obtained this way. This implies items (I) and (II) of the statement. Item (III) leads us to the situation of item (a) in Section~\ref{s1}, treated in detail in \cite{D}. For item (IV), assume without loss of generality that the missing color is $r-1=k\ell-1$. Then, exemplified in Examples~\ref{ST^3_2}-\ref{ST^3_3}, below, each 1-sphere centered at a vertex $v$ of $\Sigma_{k\ell-1}^k$ is assigned the only color $j\in[k\ell]$ distinct from $k\ell-1$ and not corresponding to any of its incident vertices. Clearly, $v$ is assigned color $j$. Thus, in general the missing color according to our assumptions is $k\ell-1$.
\end{proof}

\begin{example}\label{olvid}
The graph $ST^2_2$ of Example~\ref{ex1} has also the totally efficient coloring depicted on the lower left of Figure~\ref{fig1}, where $\Sigma_1^2=\{0011,1100\}$ is color blue, as is $E_1^2=\{(0101,1001),(0110,1010)\}$; $\Sigma_2^2=\{0101,1010\}$ is color green, as is $E_2^2=\{(0110,1100),$ $(0011,1001)\}$; and $\Sigma_3^2=\{0110,1001\}$ is color red, as is $E_3^2=\{(0011,1010),(0101,1100)\}$.\end{example}

\begin{remark}\label{occur} The total coloring of $ST^2_k$ will be referred to as its {\it color structure}. 
The $k2^{k-1}$ copies of $ST^2_{k-1}$ in $ST^2_k$ whose disjoint union is $ST^2_k\setminus\Sigma_i^k\setminus E_i^k$ inherit each a color structure that generalizes that of the 3-colored 6-cycles in $ST^2_3\setminus\Sigma_5^3$ and is similar to the color structure of $ST^2_{k-1}$.
\end{remark}

\begin{figure}[htp]
\hspace*{2cm}
\includegraphics[scale=0.77]{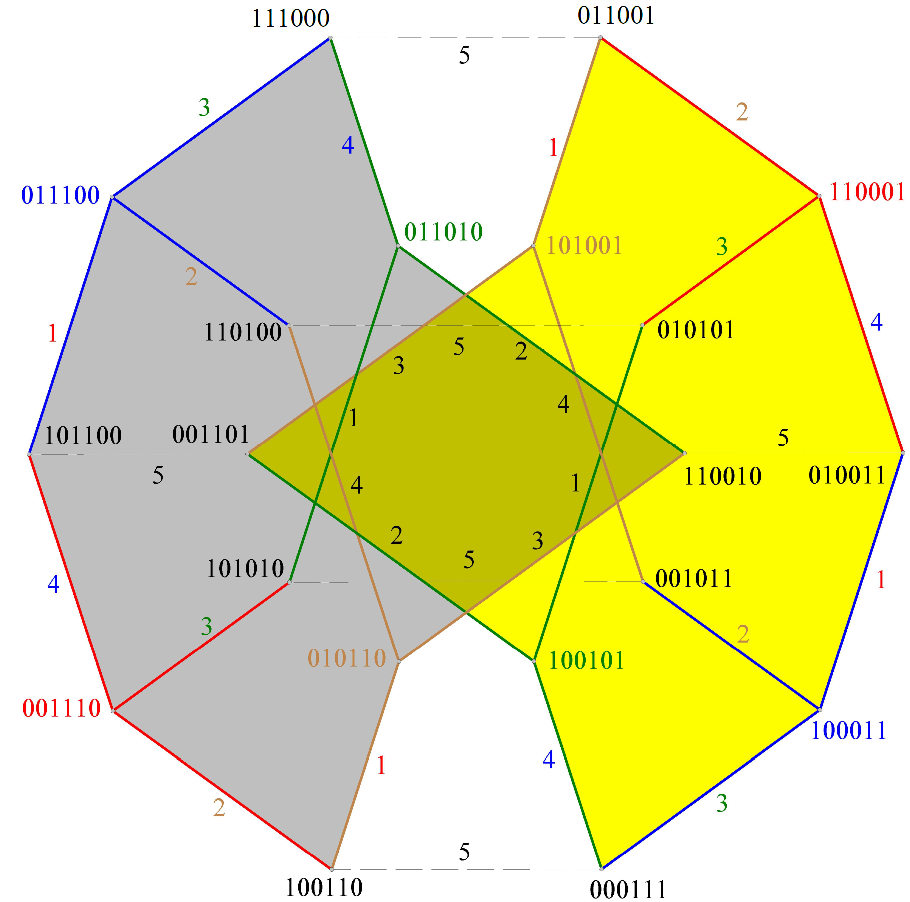}
\caption{The Desargues graph $ST^3_2$ revisited.}
\label{fig2}
\end{figure}

\begin{example}\label{ST^3_2} The Desargues graph $ST^3_2$ of Example~\ref{ex2} will now also be shown to contain the E$^2$-set $\Sigma_5^2$
depicted in Figure~\ref{fig2}, where $\Sigma_5^2$ is composed by the pairs of vertices in display (\ref{centered}), with notations colored as indicated in parentheses (in contrast to the E$^3$-set depicted on the right of Figure~\ref{fig1}),  that is: the vertices representing those 6-tuples with first, ($v_0$), and last, ($v_5$), entries having a common value, (either 0 or 1).
\begin{eqnarray}\label{centered}\begin{array}{lll}
011100&100011&(blue)\\
010110&101001&(hazel)\\
001110&110001&(red)\\
011010&100101&(green)
\end{array}\end{eqnarray}
 In fact, each subgraph $K_{1,3}$ induced by a D$3$Ss $\Sigma_5^2(v)$ of a vertex $v=v_0v_1\cdots v_4v_5=iv_1\cdots v_4j$ of $ST^3_2$ wrt $\Sigma_5^2$, ($i\ne j$, with notation in black), has its two edges other than their dashed black-colored edges in two different colors. Moreover, the edge colors in this representation of $ST^3_2$ define eight monochromatic copies of $K_{1,3}$ centered at the vertices of the form $w=w_0w_1\cdots w_4w_5=jw_1\cdots w_4 j$ in thick colors red, hazel, green and blue, as indicated in display (\ref{centered}) (two vertex-disjoint monochromatic copies of $K_{1,3}$ per color), illustrating that each $v\in ST^3_2\setminus\Sigma_5^2$ is the intersection of $\ell-1=$ two 1-spheres, each such 1-sphere contributing just one edge incident to $v$, where the two intervening edge colors are distinct.  
Those monochromatic subgraphs $K_{1,3}$ appear in ``opposing'' pairs, allowing to raise the question of how many colors are necessary to color such edge partitions.
However, observe that: {\bf(a)} the 1-factor $E_5^2$ conformed by the dashed black-colored edges induces $V(ST^3_2)\setminus\Sigma_5^2$; {\bf(b)} we can assign the said thick edge colors red, hazel, green and blue to the numbers $j=1$, 2, 3 and 4, respectively, so that the sole thick edge incident to any leaf vertex of a monochromatic copy of $K_{1,3}$ outside such a copy of $K_{1,3}$ has color $j$ if and only if $j$ is the color number of such $K_{1,3}$; 
{\bf(c)} $ST^3_2\setminus E_5^2$ has two components, represented by light-gray and yellow area colors, for $j=0$ and $j=1$, respectively; {\bf(d)} There are $2\times{{6}\choose{2}}=12$ connected components in $ST_3^2\setminus\Sigma_5^2\setminus E_5^2$, which consist of isolated vertices.
\end{example}

\begin{example}\label{ST^3_3} $ST^3_3$ has $\frac{9}{(3!)^3}=1680$ vertices and regular degree 6. Its vertices $v=v_0\cdots v_8$ have three entries $v_i$ with value $j$, for each of $j=0,1,2$. Those of them with $v_0=v_8$ form an E-set with 420 vertices, where $v_0=v_8=j$ happens for 140 vertices, for each of $j=0,1,2$. The remaining value $j$ in each such case (apart from those two occurring at $v_0$ and $v_8$) occurs at the entries $i=1,\ldots,7$ in $\frac{6!}{(3)^2}=20$ different ways. For instance, consider the "dashed black" edge $(v,w)=(101122200,001122201)$ in color 8 and the neighbors $v^1=011122200$ and $v^2=001112210$ of $v$, (resp. $w^1=100122201$ and $w^2=1011022201$ of $w$). Each of these four vertices determines a component of $ST^3_3\setminus\Sigma_8^3\setminus E_8^3$. In fact, there are $3\times{{7}\choose{2}}=63$ such components, each isomorphic to a copy of the Desargues graph as in Figure~\ref{fig1} or~\ref{fig2}. On the other hand, $v^1$ is incident to edges of colors $1,2,3,4,5,6$ whose remaining incident vertices are: $$101122200, 110122200, 111022200, 211102200, 2111202000, 211122000,$$ respectively. These in turn are adjacent via edges of color 7 to the vertices:
 $$001122210, 01022210, 011022210, 011102220, 011120220, 011122020,$$ respectively Thus,  the 1-sphere given as copy of $K_{1,6}$ centered at $v^1$ can be assigned color 7. Such a coloring is extensible to all 1-spheres of $ST^3_3$ centered at the members of $\Sigma_8^3$.

\end{example}

\end{document}